\documentclass[11pt]{amsart}

\usepackage{graphicx,amsmath,amsthm,amssymb}

\newtheorem{thm}{Theorem}

\newtheorem*{Sperner's Lemma}{Sperner's Lemma}

\newcommand{\R}{\mathbb{R}}
\newcommand{\ignore}[1]{}

\theoremstyle{definition}

\theoremstyle{remark}

\begin{document}

\title{Two-player envy-free multi-cake division}%





\author[J. Cloutier, K.L. Nyman, and F.E. Su]{
{\bf John Cloutier}\\ 
Department of Mathematics\\ 
University of California at Santa Barbara\\ 
Santa Barbara CA 93106\\
{\tt \lowercase{john@math.ucsb.edu}}
\\
\\
{\bf Kathryn L. Nyman}$^*$\\ 
Department of Mathematics\\
Willamette University\\
Salem, Oregon 97301\\
{\tt \lowercase{knyman@willamette.edu}}
\\
\\
{\bf Francis Edward Su} \\ 
Department of Mathematics\\ 
Harvey Mudd College\\ 
Claremont, CA 91711\\
{\tt \lowercase{su@math.hmc.edu}}
}

\thanks{$^*$Corresponding author}
\thanks{This research was partially supported by NSF Grant DMS-0701308.}
\thanks{{\em Journal of Economic Literature Classification}. C62, D63, D74.}
\thanks{{\em 2000 Mathematics Subject Classification}. 52B05, 91B32.} 
\thanks{Keywords:  fair division, envy-free, Sperner's lemma, polytope, labelings}

\maketitle

\begin{abstract}
We introduce a generalized cake-cutting problem in which we seek to
divide multiple cakes so that two players may get their most-preferred piece selections: a choice of one piece from each cake, allowing for the possibility of linked preferences over the cakes. 
For two
players, we show that disjoint envy-free piece selections may not exist for
two cakes cut into two pieces each, and they may not exist for three
cakes cut into three pieces each.  However, there do exist such
divisions for two cakes cut into three pieces each, and for three cakes cut into four pieces each.
The resulting allocations of pieces to players are Pareto-optimal with respect to the division.
We use a generalization of Sperner's lemma on the polytope of divisions
to locate solutions to our generalized cake-cutting problem.
\end{abstract}

\section{Introduction}
The classical cake-cutting problem of Steinhaus \cite{steinhaus} asks
how to cut a cake fairly among several players, that is, how to divide
the cake and assign pieces to players in such a way that each person
gets what he or she believes to be a fair share.  The word {\em cake}
refers to any divisible good, and the word {\em fair} can be
interpreted in many ways.  A strong notion of a fair division is an
{\em envy-free} division, one in which every player believes that their share
is better than any other share.  Existence of envy-free cake divisions
dates back to Neyman \cite{neyman}, but constructive solutions are
harder to come by;  the recent procedure of Brams and Taylor
\cite{brams-taylor} was the first $n$-person envy-free solution.
See \cite{brams-taylor-book} and
\cite{robertson-webb} for surveys.

It is natural to consider the cake-cutting question when there is more
than one cake to divide.  Given several cakes, does an envy-free
division of the cakes exist, and
under what conditions can such a division be guaranteed?

Of course, if the player preferences over each cake are independent
(i.e., additively separable), then the problem can be solved by
one-cake division methods--- simply perform an envy-free division on each cake
separately.  So the question only becomes interesting when the players have
{\em linked} preferences over the cakes, in which the portion of one cake
that a player prefers is influenced by the portion of another cake that she
might also obtain.

Consider the case of two players, Alice and Bob, dividing two cakes.
Suppose that each cake is to be cut into two pieces by a single cut,
and players may choose one piece from each cake (called a {\em piece
selection}).  Note that there are 4 possible piece selections,
only 2 of which will be chosen by Alice and Bob.  See Figure
\ref{pieceSelection}.  

We would like the 2 chosen piece selections to be {\em disjoint}, i.e., have
no common piece on either cake, if Bob and Alice are to avoid a fight.
Also, we would like their piece selections to be {\em
envy-free}, which means that each player would not want to trade
their piece selection for any of other 3 possible piece selections.
We remark that this notion of {\em envy-free piece selection} is stronger than the notion of 
an {\em envy-free allocation}, in which after pieces are allocated to players,
a given player is only comparing their piece selection to what other players actually receive
(in this example, there is just 1 other piece selection that's been assigned to a player).
Our stronger notion of envy-freeness for piece selections ensures that the allocation 
is Pareto-optimal  over all possible piece selections in that division.

Does there always exist a division of the cakes (by single cuts) so
that the players have disjoint envy-free piece selections?

\begin{figure}[htpb]
\begin{center}
\includegraphics[height=1in]{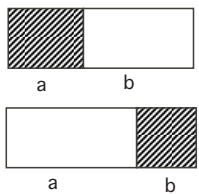}
\end{center}
\caption{The shaded and unshaded pieces represent disjoint piece
  selections $ab$ and $ba$, respectively.}
\label{pieceSelection}
\end{figure}

Consider the following scenario as an example of this generalized
cake-cutting question.  A company employs two part-time employees Alice and Bob
who work half-days (either a morning or an afternoon shift)
two days each week.  Between them,
they should cover the entire day on each of those two days.
Now, Alice and Bob may have certain preferences (such as preferring
a morning or afternoon shift) and such preferences may be linked (e.g., Alice
might highly prefer the same schedule both days, whereas Bob might
prefer to have a morning shift on one day if he has the afternoon on
another).

Their boss would like to account for both of their preferences.
Suppose that she has a certain fixed amount of salary money per
day to divide up between the morning and afternoon shifts.
Is there always a way to split the daily salary pool among the shifts 
so that Alice and Bob will choose different shifts?
If so, is there a method to find it?  If not, why not?  Here, the
cakes are the salary pools for each day, and
the pieces of cake are the shift assignments along with their salary.

In this paper we examine what can be said about envy-free assignments
in situations like this one. In particular, we show that for two
cakes, single cuts on each cake are insufficient to guarantee an
envy-free allocation of piece selections to players; there exist
preferences that Alice and Bob could have for which it is impossible
to cut the cakes so that they would prefer disjoint piece selections.
However, if one of those two cakes is divided into 3 pieces, then,
somewhat surprisingly, there does exist an envy-free allocation of
piece selections to players (with one unassigned piece).

Similarly, for three or more cakes, cutting each cake into three
pieces is not sufficient to guarantee the existence of an envy-free
allocation of piece selections to two players.  However, if each cake
is cut into four pieces, we find that such an allocation is always
possible.
These results are reminiscent of the cake-cutting procedure of Brams and Taylor \cite{brams-taylor} in which the authors show that for more than two players, the number of pieces required to ensure an envy-free division is strictly larger than the number of players.  Here, however, we find that extra pieces are necessary even in the two-player case.

The key idea here is to view the space of possible divisions as a {\em
  polytope}; we triangulate this polytope and label each vertex in a
way that reflects player preferences for piece selections in the
division that the vertex represents.  The labels satisfy the
conditions of a generalization of Sperner's lemma \cite{sperner} to
polytopes; its conclusion suggests a division and 
a disjoint allocation of piece selections to players that is envy-free, and Pareto-optimal with respect to all possible piece selections in that division.

\section{The Polytope of Divisions and the Polytopal Sperner lemma}

Sperner's lemma \cite{sperner} is a combinatorial analogue of the
Brouwer fixed point theorem, one of the most important theorems of
mathematics and economics.  Constructive proofs of Sperner's lemma can
be used to prove the Brouwer theorem in a constructive fashion; such
proofs have therefore found wide application in game theory and
mathematical economics to locate fixed points as well as Nash
equilibria (e.g., see \cite{todd, yang} for surveys).

More recently, Sperner's lemma and related combinatorial theorems have 
been used to show the existence of
envy-free divisions for a variety of {\em fair division problems}, 
including the classical cake-cutting question, e.g., see \cite{su} and \cite{simmons-su}.
In this paper we will use a recent generalization of Sperner's lemma
to polytopes \cite{deloera-peterson-su} to address our generalized
cake-cutting question.

Before describing Sperner's lemma or its generalization, we review
some terminology.  A {\em $d$-simplex} is a generalization of a
triangle or tetrahedron to $d$ dimensions: it is the convex hull of
$d+1$ affinely independent points in $\R^{d}$.  A {polytope} $P$ in
$\R^{d}$ is the convex hull of $n$ points $v_{1}, v_{2}, ... , v_{n}$,
called the {\em vertices} of the polytope.  We call an $n$-vertex,
$d$-dimensional polytope an {\em $(n,d)$-polytope}.  A {\em face} of a
polytope is the convex hull of any subset of the vertices of that
polytope; a $(d-1)$-dimensional face of $P$ is called a {\em facet}.

A {\em triangulation} $T$ of $P$ is a finite collection of distinct
simplices such that: (i) the union of all the simplices in $T$ is $P$,
(ii) the intersection of any two simplices in $T$ is either empty or a
face common to both simplices and (iii) every face of a simplex in $T$
is also in $T$.  The vertices of simplices in $T$ are called {\em
  vertices of the triangulation} $T$.

A {\em Sperner labeling} of $T$ is a labeling of the
vertices of $T$ that satisfies these conditions: (1) all vertices of
$P$ have distinct labels and (2) the label of any
vertex of $T$ which lies on a facet of $P$ matches the label of one of
the vertices of $P$ that spans that facet.  A {\em full cell} is any
$d$-dimensional simplex in $T$ whose $d+1$ vertices possess distinct
labels.  See Figure \ref{sperner} for examples of Sperner-labeled
polygons (with full cells shaded).

Sperner's lemma \cite{sperner} states that
any Sperner-labeled triangulation of a simplex
contains an odd number of full cells.  The generalization which will
be of use to us is:

\begin{thm}[DeLoera-Peterson-Su]
Any Sperner-labeled triangulation of an $(n,d)$-polytope $P$
must contain at least $(n-d)$ full cells.
\label{polytopalSperner}
\end{thm}

For instance, the Sperner-labelled pentagon in Figure \ref{sperner}
has at least $5-2=3$ full cells.

\begin{figure}[htpb]
\begin{center}
\includegraphics[height=2.5in]{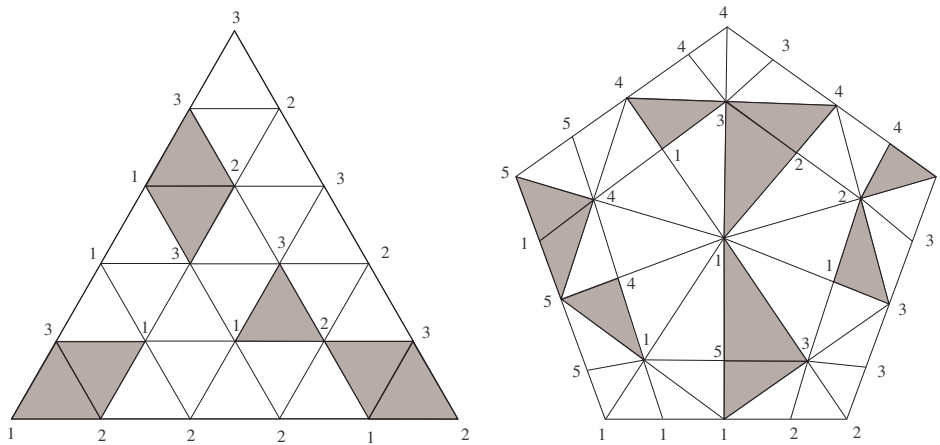}
\end{center}
\caption{A Sperner-labeled triangle (2-simplex) and pentagon
  ($(5,2)$-polytope).  Full cells are shaded.}
\label{sperner}
\end{figure}

The polytope that will interest us is the {\em polytope of divisions},
which we describe presently.  Suppose that we have $m$ cakes, each of
length 1, and that each cake is to be cut into $k$ pieces.  Let the
length of the $j$-th piece of the $i$-th cake be denoted by $x_{ij}$.
Then for all $1 \leq i \leq m$ we have
$$x_{i1} + x_{i2} + ... + x_{ik} = 1.
$$

Each such division can be represented by an $m \times k$ matrix:

\[ \left( \begin{array}{cccc}
 x_{11} & x_{12} & ... & x_{1k} \\
x_{21} & x_{22} & ... & x_{2k} \\
\vdots & \vdots & \ddots & \vdots \\
x_{m1} & x_{m2} & ... & x_{mk} \end{array} \right) \]

where each row sum is 1.

Now, suppose that each cake is cut in such a way that one piece is the
entire cake and all other pieces are empty.  We will call such a
``division" of the cakes a {\em pure division}.  The matrix
representation of a pure division is one in which each row has exactly
one entry as a 1 and the rest 0.  Notice that any division may be
written as a convex linear combination of the pure divisions.  Thus,
it is natural to view the space of divisions as a polytope $P$ with
the pure divisions as its vertices.  We call $P$ the {\em polytope of
  divisions}.  From the matrix representation of this polytope we see
that the space of divisions is of dimension $m(k-1)$, since in each of
the $m$ cakes, the length of the last piece is determined by the
lengths of the first $k-1$ pieces.  Also $P$ has $k^{m}$ vertices, one
for each pure division.  Thus $P$ is an $(n,d)$-polytope with $d =
m(k-1)$ and $n = k^{m}$.

\section{The Owner-Labeling and the Preference-Labeling}
\label{SpernerMethod}

We now describe how to locate envy-free piece selections using the
polytope of divisions $P$.  Suppose that there are two players, $A$ and $B$.  
Let $T$ be a triangulation of $P$.
Next, label $T$ with an {\em owner-labeling}:
label each vertex in $T$ with either $A$ or $B$.  We will want this
owner-labeling to satisfy the condition that
each simplex in $T$ has roughly the same number of $A$ and $B$
labels.  

In general, for any number of players, call an owner-labeling {\em
  uniform} if in each simplex, the number of labels for each player
differs by at most one from any other player.  So for two players $A$
and $B$, if a simplex $\sigma$ in $T$ has $n$ vertices, a uniform
labeling would assign labels $A$ and $B$ to $n/2$ vertices each if
$n$ is even, and to at least $(n-1)/2$ vertices each if $n$ is odd.

We claim that any polytope has a triangulation of arbitrarily small
mesh size that can be given a uniform owner-labeling.  In particular,
this may be accomplished by choosing a triangulation $T$ of the
required mesh size, and {\em barycentrically subdividing} $T$.  A {\em
  barycentric subdivision} takes a $k$ simplex $\sigma$ of $T$, and
replaces it by $k!$ smaller $k$-simplices whose vertices are the
barycenters of an increasing saturated chain of faces of $\sigma$.
Since each smaller $k$-simplex contains exactly one barycenter in each
dimension, we may assign the even-dimensional barycenters to player
$A$ and the odd-dimensional barycenters to $B$.  This owner-labeling
is uniform for players $A$ and $B$.  See Figure \ref{owner} for an
example of a barycentric subdivision with a uniform owner-labeling.
(If there are more than two players, then cyclically rotate player
assignments by dimension).

\begin{figure}[htpb]
\begin{center}
\includegraphics[height=1.5in]{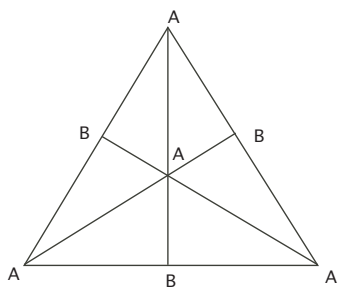}
\end{center}
\caption{A barycentric subdivision with uniform owner-labeling for players $A$ and $B$.}
\label{owner}
\end{figure}

Now that the ownership has been assigned to the vertices of $T$, we
shall construct a {\em preference-labeling} of $T$.  

Given a division 
(not necessarily pure) of the $m$ cakes, a {\em piece selection} is a
choice of one piece from each cake.
We say that a player {\em prefers}
a certain piece selection if that player does not think that any other piece
selection is better.  We make the following three assumptions about
preferences:

(1) {\em Independence of preferences}.  A player's preferences depend
    only on that player and not on choices made by other players.

(2) {\em The players are hungry}.  A player will always prefer a
    nonempty piece in each cake to an empty piece in that cake
    (assuming the pieces selected in the other cakes are fixed).
    Hence a preferred piece
    selection will contain no empty pieces.

(3) {\em Preference sets are closed}.  If a piece selection is
    preferred for a convergent sequence of divisions, then that piece
    selection will be preferred in the limiting division.

Note that with these assumptions about player preferences, a given
player always prefers at least one piece selection in any
division and may prefer more than one if that player
is indifferent between piece selections.

We now construct the preference-labeling of $T$.  Notice that a vertex
$v$ of $T$ is just a point in $P$ and so it represents a division of
the cakes.  Ask the owner of $v$ (who is either $A$ or $B$):
$$ \mbox{``which piece selection do you prefer in this division?''}
$$
Label $v$ by the answer given.  Since every
piece selection corresponds to a pure division in a natural way
(namely, each may be thought of as a choice of one piece from each cake), 
we obtain a preference-labeling of the vertices of
$T$ by pure divisions.

This new labeling is, in fact, a Sperner labeling.  To see why, note
that each vertex of $P$ is a pure division, and condition (2)
ensures that the player who owns that vertex will select the unique
non-empty piece from each cake.  So, every vertex of $P$ will be
labeled by its corresponding pure division.  A vertex $v$ of $T$ that
lies on a facet of $P$ is a strict convex linear combination of the
subset of vertices of $P$ that span that facet.  Thus, the division
represented by $v$ is a strict convex linear combination of the subset
of pure divisions that are represented by the vertices that span the
facet.  So, if the $i$-th piece (on any cake) of the division
represented by $v$ is empty, then each of those pure divisions must
have an empty $i$-th piece as well.  So condition (2) guarantees that
the owner of $v$ will never prefer a piece selection that corresponds
to a pure division that is not on the same facet as $v$.  Thus, the
preference-labeling is a Sperner labeling.

The polytopal Sperner lemma shows that there exist $(n-d)$ $d$-dimensional,
full cells in $T$.  The owner labeling of each of these
cells is uniform.  Hence, a full cell represents $d+1$ similar
divisions in which players $A$ and $B$ choose different piece selections.

Now, if we repeat this procedure for a sequence of finer and finer
triangulations of $P$, we would create a sequence of smaller and
smaller full cells.  Since $P$ is compact, there must be a
convergent subsequence of full cells that converges to a single point.
Since each full cell in the convergent
subsequence also has a uniform owner labeling and since there are only
finitely many ways to choose a piece selection, there must be an
infinite subsequence of our convergent sequence for which the piece
selections of each player remain unchanged.  Condition (3) guarantees
that the selections will not change at the limit point of these full
cells.  So, at this limit point, the players choose different piece
selections just as they did in the cells of the sequence.

In this way, we may find a division for which both players choose {\em
  different} piece selections, that is, selections where the players
choose different pieces on {\em at least one} cake.  More generally,
with any number of players, the same arguments (using a
uniform-labelling over multiple players) show:

\begin{thm}
\label{notdisjoint}
Given $p$ players and $m$ cakes, if each cake is cut into $k$ pieces
each and $k\geq 1+(p-1)/m$, then there exists a division of the cakes for
which each player chooses {\em different} (though not necessarily
disjoint) piece selections.
\end{thm}

In particular, this theorem holds if each cake is cut into $k=p$
pieces.  The bound on the number of pieces comes from the fact that
the number of vertices of the simplex must be at least as large as
the number of players so that each player owns at least one vertex of
the full cell.
The conclusion says that the piece selections chosen are {\em envy-free}: 
no one would trade their piece selection for any other.
However, this theorem does not help us allocate the pieces, because
two different players may have conflicting piece selections if they
chose the same piece on a particular cake.  

So we wish to find divisions for which the players make {\em disjoint} envy-free piece
selections, i.e., selections where the players would choose non-conflicting 
pieces on {\em every} cake.  We remark that such an allocation would be Pareto-optimal with 
respect to all possible piece selections in that division.
This is because each player would have chosen, over {\em all piece selections}, what he/she most prefers.  So there would be no other allocation of pieces of a given division 
in which players could do better.

We explain this in contrast to an interesting example of Brams, Edelman, and Fishburn \cite{brams-edelman-fishburn}, who consider a single cake cut into 6 pieces in which 3 players choose 2 pieces each.  They exhibit a division and allocation of a pair of pieces to each player in such a way that the allocation is envy-free but not Pareto optimal, because for the given 6-piece division there is another allocation in which some players are better off.  Although theirs was an {\em envy-free allocation} (involving comparisons of 3 allocated pairs of pieces), 
those pairs were not {\em envy-free piece selections} (involving comparisons of 30 possible pairs of pieces).  An division into envy-free piece selections would by nature yield a Pareto-optimal allocation with respect to the given division.

In what follows we explore the existence of disjoint envy-free piece selections.

\section{Divisions of Two Cakes}
\subsection{Two cakes, two pieces each}\
\label{sec2ck2pc}
We now explore what happens when two cakes, each cut into two pieces,
are divided among two players:  $A$ and $B$.
We show:

\begin{thm}
Given 2 players and 2 cakes, there does not necessarily exist a
division of the cakes into 2 pieces each that contains disjoint envy-free
piece selections for those players.
\label{thm:2cakes2pieces}
\end{thm}

\begin{proof}
Let $P$ be the polytope of divisions.  In this case, $P$ is a
$(4,2)$-polytope, i.e., a square.  We label the pure divisions as
follows: $aa$, $ab$, $bb$, $ba$, where, for instance, $aa$ represents
the pure division in which the left pieces of Figure \ref{pieceSelection} are
the entire cakes.

Recall that the same names $aa$, $ab$, $ba$, $bb$ can be used to
reference piece selections.  Also note that $aa$ and $bb$ represent
disjoint piece selections, and so do $ab$ and $ba$. 

Player A's preferences can be described by a {\em cover} of the square
$P$ by four sets, labeled by piece selections:$aa$, $ab$, $ba$, $bb$.
This is accomplished by placing point $p$ in $P$ in the set $xy$ if in
the division of cake represented by $p$, player $A$ would prefer piece
selection $xy$.  Some points may be in multiple sets (representing
indifference between piece selections).  Assumption (3) about
preferences (see previous section) ensures that the four sets are
closed.  The union of the four sets contains $P$.  Assumption (2)
about preferences ensures that the pure divisions are in the sets that
bear their own name.

In a similar fashion, player $B$ also has a cover by four sets that
describes his preferences.  To find a division that allows for
disjoint envy-free piece selections, we seek a point $p$ in $P$ such
that $p$ is in one player's $aa$ set and the other's $bb$ set, or in
one player's $ab$ set and the other's $ba$ set.

Now consider the player preferences shown in Figure \ref{2ck2pc},
where we assume that each set is labeled by the pure division that it
contains.  A visual inspection shows that there is no such point
$p$ as described above.
\end{proof}

\begin{figure}[htpb]
\begin{center}
\includegraphics[height=2.5in]{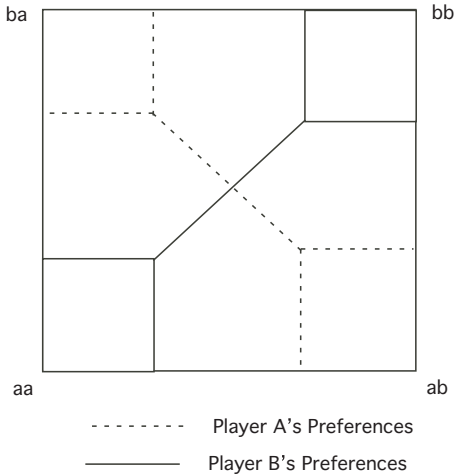}
\end{center}
\caption{Conflicting preferences for 2 players, 2 cakes and 2 pieces per cake.}
\label{2ck2pc}
\end{figure}

Therefore, it is not necessarily possible to divide two cakes into
two pieces in such a way that the two players will choose different pieces
on each cake, i.e., find disjoint piece selections.  This result
may seem somewhat unexpected, since if preferences are not linked,
we can divide each individual cake in such a way that both
players would be satisfied with their piece on each individual cake.

An interpretation of Figure \ref{2ck2pc} provides an example of linked
preferences. Player $A$ generally seems to either want the left hand or
the right hand pieces of both cakes, and player $B$ generally wants both
a left and a right hand piece (in both cases, as long as the pieces are not too small). 
In our running example from the introduction, the preferences
in Figure \ref{2ck2pc}
would correspond to Alice strongly desiring to work either both morning or both
afternoon shifts, and Bob strongly desiring to work one morning and one
afternoon shift.  There may then be no division of the salary pool on
each day that will induce them to take disjoint shifts.

One may ask why the arguments of the previous section do not extend here.  The answer is that while in each triangulation we can ensure an $A$-$B$ edge (in owner-labels) that corresponds to endpoints with different preference-labels, we cannot ensure that the corresponding piece selections are disjoint.  Indeed we see from Figure \ref{2ck2pc} that if $P$ is triangulated very finely, there will be no $A$-$B$ edge (in owner-labels) that could possibly be $aa$-$bb$ or $ab$-$ba$ (in preference-labels) since the corresponding players' sets in the Figure are not close.

In contrast to Theorem \ref{thm:2cakes2pieces}, we may obtain a
positive result for two players when there are three players involved.

\begin{thm}  Given 3 players and 2 cakes, there exists a division of the cakes, each cut into 2 pieces, such that some pair of players has disjoint envy-free piece selections.
\end{thm}

\begin{proof}
 With three players, consider an owner-labeling in which every
 triangle has vertices labeled by $A$, $B$, and $C$, so that each
 vertex of the triangle belongs to a different player.  Note that a
 fully labeled triangle in the square of all divisions must contain a
 pair of disjoint piece selections.  So, there is a pair of players
 who prefer disjoint piece selections for two very close divisions.
 Using the limiting argument of Section \ref{SpernerMethod}, there is
 a division in which some pair of players has disjoint envy-free piece
 selections.
 \end{proof}

\subsection{Envy-free piece selections with more pieces}
\label{sec2ck3pc}
While there does not necessarily exist an envy-free allocation for two
players with each cake cut into two pieces, we can satisfy both
players with one cake cut into three pieces and the other cut in
(only) two pieces.  The proof of this fact, and later results for
three cakes, will use the following theorem:

\begin{thm}[DeLoera-Peterson-Su]
Let $P$ be an $(n,d)$-polytope with Sperner-labeled triangulation $T$.
Let $f:P \rightarrow P$ be the piecewise-linear map that takes each
vertex of $T$ to the vertex of $P$ that shares the same label, and is
linear on each $d$-simplex of $T$.  The map $f$ is surjective, and
thus the collection of full cells in $T$ forms a cover of $P$ under
$f$.
\label{cover}
\end{thm}

\begin{thm}
Given 2 players and 2 cakes, there is a division of the cakes--- one
cut into 2 (or more) pieces, the other cut into 3 (or more) pieces---
so that the players have disjoint envy-free piece selections.
\label{thm:2cakes2p3p}
\end{thm}

\begin{proof}
Let $P$ be the (6,3)-polytope of divisions of the two cakes.  In this
case, $P$ is a triangular prism with vertices corresponding to the
pure divisions $aa, ab, ac, ba, bb$, and $bc$.  Notice the 1-skeleton of $P$
can be interpreted as a graph in which piece selections
that conflict appear as labels on adjacent vertices of $P$ (see
Figure \ref{2pieces3pieces}).

\begin{figure}[htpb]
\begin{center}
\includegraphics[height=2in]{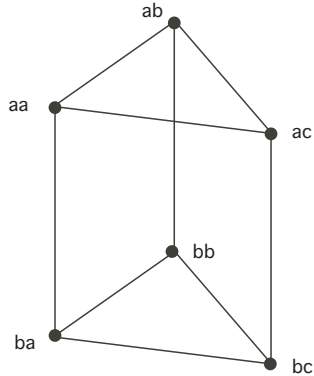}
\end{center}
\caption{The polytope of divisions for 2 cakes divided into 2 and 3 pieces.}
\label{2pieces3pieces}
\end{figure}

   Let $T$ be a triangulation of $P$ with a uniform owner-labeling.
   By Theorems \ref{polytopalSperner} and \ref{cover}, there
   exists a fully-labeled 3-simplex $\sigma$ whose image $f(\sigma)$,
   under the map $f$ of Theorem $\ref{cover}$, is one of the simplices
   of the cover of $P$.  Therefore, $f(\sigma)$ is non-degenerate and
   its four vertices do not lie on a common face of $P$.  In this
   case, one can verify that one vertex $v$ of $f(\sigma)$ must be
   non-adjacent, in the 1-skeleton of $P$, to two other vertices,
   $w$ and $y$.  This means $w$ and $y$ correspond to piece selections
   that are disjoint from the piece selection of $v$.  Without loss of
   generality, if $A$ owns $v$, $B$ must own at least one of $w$ and
   $y$ because the owner labeling is uniform.  Therefore, using the
   methods of Section \ref{SpernerMethod}, we can find an envy-free
   allocation.

   If the two cakes are divided into more than 2 or 3 pieces
   respectively, we can find an envy-free allocation simply by
   restricting our attention to divisions of the cake in which the
   extra pieces are empty, and using the above results.
   (But other envy-free divisions may exist as well.)
\end{proof}

\section{Three Cakes}
\subsection{Three cakes, three pieces each}
As shown in Section \ref{sec2ck2pc}, dividing two cakes into two
pieces each was insufficient to guarantee a division with envy-free
piece selections for two players.  We shall see that for two players,
dividing three or more cakes into three pieces each is also
insufficient, but cutting three cakes into four pieces each guarantees
a division with envy-free piece selections.  We begin by examining
divisions of three cakes into three pieces each.

\begin{thm}
Given 2 players and 3 (or more) cakes, there does not necessarily exist a
division of the cakes into 3 pieces each that contains disjoint envy-free
piece selections for the two players.
\label{preferences}
\end{thm}

\begin{proof}
We will let $a, b,$ and $c$ designate the
leftmost, center, and rightmost piece of each cake, such that a choice
of $bba$, for example, will refer to a player choosing the second
piece in cakes one and two, and the first piece in cake three (see
Figure \ref{bbaSelection}).
\begin{figure}[htpb]
\begin{center}
\includegraphics[height=1.75in]{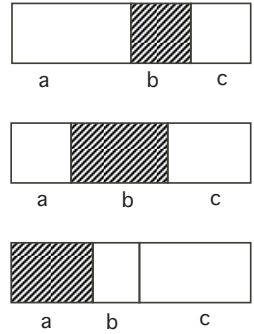}
\end{center}
\caption{A piece selection with two pieces of the same type ($bba$).}
\label{bbaSelection}
\end{figure}

We now construct preferences for two players, Alice ($A$) and Bob
($B$), for which there does not exist a division with envy-free
disjoint piece selections.  Fix some small $\epsilon >0$.  Let $A$
prefer piece selections according to the following broad categories in
descending order of preference:

\begin{enumerate}
\item three pieces of the same type (i.e., $aaa$, $bbb$, $ccc$),
\item two pieces of the same type (e.g., $aba$, $ccb$),
\item three pieces all of different type (e.g., $abc$).
\end{enumerate}
Let $B$'s preferences be the reverse:
\begin{enumerate}
\item three pieces all of different type,
\item two pieces of the same type,
\item three pieces of the same type.
\end{enumerate}
Neither player will accept a piece selection if any of its pieces are
of size less than $\epsilon$.  If two or more piece selections are
available in a given preference category, with all pieces greater than
$\epsilon$ in size, the player chooses the option with the greatest
total size (if two choices have the same total size, choose the
lexicographic first option).  Players only move to a lower ranked
preference category if all options in a higher category contain a
piece of cake with size less than $\epsilon$.

We show that for any set of pieces that player $A$ chooses, player $B$
prefers a piece selection which is not disjoint from $A$'s.  Suppose
$A$ chooses three pieces of the same type; without loss of generality,  
say $A$ chooses $aaa$.  Thus, each piece $a$ has size greater than $\epsilon$.
If $B$'s piece selection were disjoint from $aaa$, it would contain no
$a$, and hence would contain at least two $b$'s or two $c$'s.
However, replacing one of those repeated letters with an $a$ would
result in a piece selection more desirable to $B$.

Next, suppose $A$'s piece selection consists of two pieces of one type
and one piece of another; for example, say $A$ chooses
$aab$.  In the third cake, piece $a$ must have size less than
$\epsilon$, otherwise $A$ would have chosen $aaa$.  Therefore, to not
conflict with $A$'s choice, $B$'s piece selection must contain piece
$c$ in the last cake.  Non-conflicting choices for $B$ are $bbc$,
$bcc$ or $cbc$; however, the $a$ pieces in the first two cakes have
size greater than $\epsilon$, so the piece selections $abc$, $bac$ or
$abc$ would be preferred by $B$ to the previous options, respectively.
These conflict with $A$'s piece selection in one of the first two
cakes.

  Finally, suppose $A$ chooses three different types of pieces, for
  example, $abc$.  Since $A$ would prefer $bbc$ or $cbc$ to $abc$, it
  must be that $a$ is the only piece of size greater than $\epsilon$
  in the first cake.  Therefore $B$ will also choose piece $a$ from
  the first cake, and $A$'s and $B$'s piece selections will not be
  disjoint.

We have shown that a division with disjoint envy-free piece selections
for two players need not exist for three cakes divided into three
pieces each.  It is easy to see that the same is true for four or more
cakes.  Simply have the players adopt the preferences described above
on the first three cakes.
\end{proof}

\subsection{Three cakes, four pieces each}
We now show that if, instead, the three cakes are each divided into
four or more pieces, we can always find a division with envy-free piece
selections for two players.

\begin{thm}
Given 2 players and 3 cakes, there exists a division of the cakes into
4 (or more) pieces each that contains disjoint envy-free piece selections for
both players.
\end{thm}

As before, the idea of the proof is to use the Polytopal Sperner Lemma (Theorem \ref{polytopalSperner}) on the $9$-dimensional space of divisions $P$.  However, our task is made much simpler by focusing on a particular full cell: the one that covers the center of $P$.

\begin{proof}
Let $P$ be the (64,9)-polytope in $\R^{12}$ of divisions of 3
cakes into 4 pieces each, and let $T$ be a triangulation of $P$ with a
uniform owner labeling.  By Theorem \ref{cover}, the fully-labeled
cells of $T$ cover $P$ under the map $f$.  Hence, for at least one
fully-labeled 9-simplex $\sigma$, $f(\sigma)$ contains the center of
$P$.  That is, there exist weights $a_i>0$, such that $Q =
\sum_{i=1}^{10} a_i M_i$ where the $M_i$ are the matrices of the pure
divisions corresponding to the vertices of $f(\sigma)$, and $Q$ is the
$k \times m$ matrix in which every entry is 1/4 (this is the center of
$P$).

\begin{figure}[htpb]
\begin{center}
\includegraphics[height=2.25in]{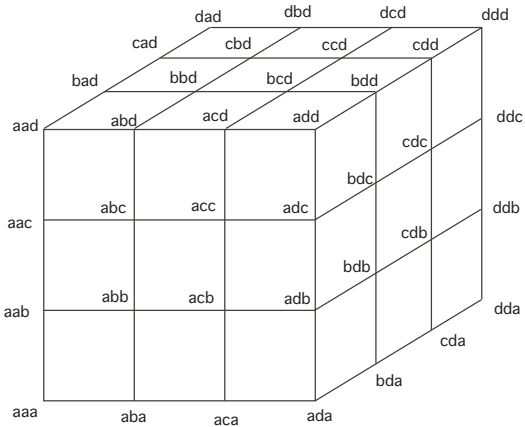}
\end{center}
\caption{The vertices of $P$, the pure divisions, arranged in a 
$4\times4\times4$ grid.  These also may be thought of as piece selections, and two piece selections are disjoint if and only if they do not lie on the same grid plane.}
\label{12planes}
\end{figure}

It will be convenient to visualize the vertices of $P$ arranged on a
$4\times4\times4$ grid, since there are four pieces $\{a,b, c, d\}$
that may be selected for each of three cakes (see Figure
\ref{12planes}).  

Note that there are four planes in each of the 3 orthogonal grid
directions; in the rest of the proof, a {\it plane} will refer to one
of these 12 special planes.  For example, the bottom horizontal plane
contains all vertices of $P$ corresponding to pure divisions in which
piece $a$ is picked in the 3rd cake.  In fact, any given plane
corresponds to pure divisions in which a particular piece in a
particular cake is chosen.  The matrices corresponding to such
divisions all have a 1 in the same entry.  Since $Q$ is a weighted
average of $0$-$1$ matrices $M_i$ that represent pure divisions, and
matrices with a $1$ in a particular entry correspond to vertices in
our grid that lie on a common plane, the total weight of vertices of
$f(\sigma)$ that lie on a given plane of our grid must be $1/4$.  All of
our arguments will be based on this fact.

A vertex in our $4\times4\times4$ grid may also be thought of as a
piece selection.  By doing so, we see that two piece selections lie on
a common plane if and only if they are not disjoint.  

Recall that $\sigma$ denotes the full cell whose image $f(\sigma)$ contains the center of $P$.
Consider a graph $G$ whose nodes are the labels of $\sigma$, and in which two nodes
are adjacent in $G$ if and only if they represent disjoint piece selections, i.e., if and only if their labels
do not lie on a common plane in Figure \ref{12planes}.

{\em We now show, in all arguments that follow, that $G$ has a connected component of size at least $6$.} Since the uniform owner labeling of a 9-simplex has exactly 5 vertices owned by $A$ and 5 by $B$, this would imply the desired conclusion: there is an $A$-$B$ edge of $\sigma$ with disjoint preference-labels.

Call a configuration of 4 vertices in the $4\times4\times4$ cube a
{\it diagonal} if each plane contains exactly one vertex.  These
corresponding piece selections are all mutually disjoint, and so they
form a size-4 clique in $G$.  Furthermore, any other vertex in the
cube lies on exactly three planes, and so must be disjoint from one of
the diagonal vertices.  Thus, a set of vertices that contains a
diagonal is connected in $G$.

We now focus attention on the number of vertices of $f(\sigma)$ on each plane of the
cube and their associated weights $a_i$.  Consider first the case in
which some plane contains only one vertex, $v_j$; hence $a_j =1/4$.
Vertex $v_j$ is adjacent, in $G$, to every other vertex with
non-zero weight, since $v_j$ cannot lie on a plane with any other
weighted vertex (otherwise the total weight of the plane exceeds $1/4$).
If there are at least five other vertices with positive weight, $G$
would have a connected component of size at least 6.  Suppose there
are five or fewer weighted vertices.  Since the
four planes in any given direction must each contain at least one weighted vertex, there must
be three planes that contain exactly one weighted vertex.  These
vertices each have weight $1/4$, and each is the only vertex on any of
the three orthogonal planes that contain it.  Therefore, any additional weighted
vertex must lie in the unique intersection of the remaining planes,
and this in turn implies a diagonal configuration of vertices.  By our
earlier argument all 10 vertices of $\sigma$ are connected in $G$.

If there is no such plane, then every plane contains at least two
vertices.  We look at the possible cases individually and show that
there is some vertex that lies on planes with at most 4 other
vertices, and is therefore disjoint from at least 5 vertices in $G$
(resulting in a connected component of size at least 6 in $G$).  We
will use the shorthand notation $(x, y, z, w)$ to indicate the number
of vertices on the 4 planes in a given direction, ordered by the
number of vertices they contain.

Case 1: (4,2,2,2), (4,2,2,2), (4,2,2,2).  Suppose that in every
direction there is one plane containing 4 vertices and three planes
containing 2 vertices each.  We claim that there exists a vertex which
does not lie on any plane containing 4 vertices. To see why this is
true, let the pairs of vertices on the three 2-vertex planes in one
direction be $\{a,b\}$, $\{c,d\}$, and $\{e,f\}$.  The combined weight
of each pair is $1/4$.  If each of these vertices appeared on a
4-vertex plane in at least one of the other two directions, the sum of
the weights of those two planes would be at least the combined weight
of the six vertices $\{a,b,c,d,e,f\}$, or $3/4$, and so at least one
of the planes would have weight greater than $1/4$, a contradiction.
Therefore, there is a vertex which lies only on 2-vertex planes in
each direction.  This vertex is not disjoint from at most 3 other
vertices, and so is disjoint from at least 6.

Case 2: (4,2,2,2), (4,2,2,2), (3,3,2,2).  At most eight of the ten
vertices of $\sigma$ lie on a 4-vertex plane, so at least two do not
lie on 4-vertex planes.  These have at most 4 neighbors in $G$ (it is
four if they lie on planes with 2, 2, and 3 vertices).  Thus, they are
disjoint from at least 5 vertices in $G$.

Case 3: (4,2,2,2), (3,3,2,2), (3,3,2,2).  By an argument similar to
that of Case 1, there exists a vertex which does not lie on the
4-vertex plane, and lies on at most one 3-vertex plane; otherwise, the
two 3-vertex planes in each direction must contain all 6 of the
vertices not on the 4-vertex plane --- so the total weight of the two
3-vertex planes (in either direction) would be $3/4$, which is too
large.  Therefore, there exists a vertex that lies on at most one
3-vertex plane and no 4-vertex plane.  This vertex has at most 4
neighbors in $G$ (it is four if it lies on planes with 2, 2, and 3
vertices).  Thus it is disjoint in $G$ from at least 5 vertices.

Case 4: (3,3,2,2), (3,3,2,2), (3,3,2,2).  There are twelve positions
for vertices on 2-vertex planes and only 10 vertices of $\sigma$, so
by the pigeonhole principle, one vertex must lie on at least two
2-vertex planes, hence it has at most 4 neighbors in $G$, so it is
disjoint in $G$ from at least 5 vertices.

Thus, given a fully-labeled 9-simplex $\sigma$ whose image under $f$
contains the point $Q$, a uniform owner labeling produces an $A$-$B$
edge with disjoint preference-labels.  By using the methods of Section
\ref{SpernerMethod} we are able to find a division with
disjoint envy-free piece selections for players $A$ and $B$.

As in the proof of Theorem \ref{thm:2cakes2p3p}, if the cakes are cut
into more than 4 pieces each, we can guarantee an envy-free allocation 
by restricting our attention to the case in which the extra pieces are empty.  
Other envy-free allocations may exist as well.
\end{proof}

\section{Discussion and Open Questions}

In this article, the Polytopal Sperner Lemma has given us insight into how many pieces are necessary for envy-free multiple-cake division with disjoint piece selections.  We have shown that it is
possible to divide two (respectively, three) cakes among two players
in an envy-free fashion with disjoint piece selections so long as each cake is cut into at
least three (respectively, four) pieces.  
This suggests a natural 
question for 2 players and $m$ cakes: is it always possible to find
disjoint envy-free piece selections if each cake is cut into at least $m+1$
pieces?
And can we get by with fewer pieces in some of the cakes (as we did for 2 cakes, allowing one cake to have 2 instead of 3 pieces)?

We have also shown that there may not exist envy-free divisions of
two (respectively, three) cakes among two players when the cakes are
only cut into two (respectively, three) pieces each.  In fact, with
preferences similar to those in the proof of Theorem
\ref{preferences}, one can verify that 4 cakes cut in 4 pieces each is
not sufficient to ensure an envy-free allocation among two players.  In
this case, A's preferences in order are: 4 of a kind, 3 of a kind, 2 pair, 1
pair, all different.  B's preferences are the reverse.  By extending these
preferences to $m$ cakes each cut into $m$ pieces, is it possible to show that
disjoint envy-free piece selections may not exist for two players in this situation?

Finally, how many additional pieces do we need in each cake if there are more players
who demand disjoint envy-free piece selections?  
Recall (see comments at the end of Section \ref{SpernerMethod}) that disjoint, envy-free piece selections would by nature yield Pareto-optimal allocations, even with more than two players.
Our intuition suggests that if we cut the cake into enough pieces, we should be able to satisfy all players, though there may be many extra pieces left over.   
We conclude with this open question:  for given numbers of players and cakes, we ask for the 
minimum number of pieces to divide each cake that would ensure that some division has disjoint 
envy-free piece selections.


\bibliographystyle{amsplain}

\end{document}